\documentclass [12pt] {article}
\usepackage{amsmath,amssymb,amsbsy,amsthm}
\usepackage[utf8]{inputenc}
\usepackage{graphicx}
\usepackage{multirow}
\usepackage{tikz}
\usepackage{tkz-graph}
\usetikzlibrary{shapes}
\tikzstyle{vertex}=[circle, draw, inner sep=2pt, minimum size=6pt]

\newtheorem{lemma} {\indent Lemma}
\newtheorem{theorem} {\indent Theorem}
\newtheorem{cor} {\indent Corollary}

\title{On Generalized Edge Corona Product of Graphs }%[Short Article Title]
\author{Irandokht Rezaee Abdolhosseinzadeh, Freydoon Rahbarnia \\ Ferdowsi University of Mashhad, Iran}
\begin{document}
\maketitle

\begin{abstract}
Let $G$ be a simple graph with $m$ edges and $H_i$,  $1\leq i\leq m$,
 be simple graphs too. The generalized edge corona product of graphs $G$ and $H_1,..., H_m$, 
 denoted by $G\diamond (H_1, ..., H_m)$, is obtained 
 by taking one copy of graphs $G, H_1, ..., H_m$ and joining two end vertices of $i-$th
  edge of $G$ to every vertex of $H_i$,  $1\leq i\leq m$. In this paper, some results regarding
  the $k-$distance chromatic number of generalized edge corona product of graphs are presented.
 Also, as a consequence of our results, we compute this invariant for the graphs
     $K_n\diamond (H_1, ..., H_m)$, $T\diamond (H_1, ..., H_m)$ and 
$K_{m,n}\diamond (H_1, ..., H_m)$. Moreover, the dominating set, the domination number and the independence number 
of any connected graph $G$ and arbitrary graphs $H_i$, $1 \leq i \leq |E(G)|$, 
are evaluated under generalized edge corona operation. 
\end{abstract}

\section{Introduction}
Let $G$ be a simple graph with vertex
set $V(G) = \{v_1, v_2, ... , v_n\}$ and edge set $E(G) = \{e_1, e_2, ... , e_m\}$. 
We denote the shortest distance between two vertices $v_i$ and $v_j$ in $G$ by 
$d_G(v_i, v_j)$. The \emph{eccentricity} $\varepsilon(v)$ of a vertex $v$ in  $G$
 is defined as the maximum distance between $v$ and any other vertex of $G$. 
 The maximum eccentricity over all vertices of $G$ is called the \emph{diameter} 
 of $G$ and denoted by $D(G)$ and the minimum eccentricity among all vertices 
 of $G$ is called the \emph{radius} of $G$ and denoted by $r(G)$.
 
 Let $G$ be a simple graph with $m$ edges and $H_1, H_2,…, H_m$ be $m$ simple graphs. The \emph{generalized edge corona product}, denoted by $G\diamond (H_1, ...,H_m)$, is the graph obtained by taking one copy of graphs $G, H_1, H_2,…,H_m$ and then joining two end-vertices of the $i-$th edge $e_i$ of  of $G$ to every vertex of $H_i$  for $1\leq i\leq m$ \cite{Lu}.  In particular, if 
$H_1, ..., H_m$ are isomorphic graphs, then the generalized edge corona becomes
 to the well-known \emph{edge corona product} of two graphs $G$ and $H$ denoted
  by $G\diamond H$ \cite{Chi, Ho, Re}. As an example of edge corona of two graphs see Figure \ref{fig1}.

% The position of fig1
\begin{figure}[h]
\center{
\includegraphics[scale=0.3]{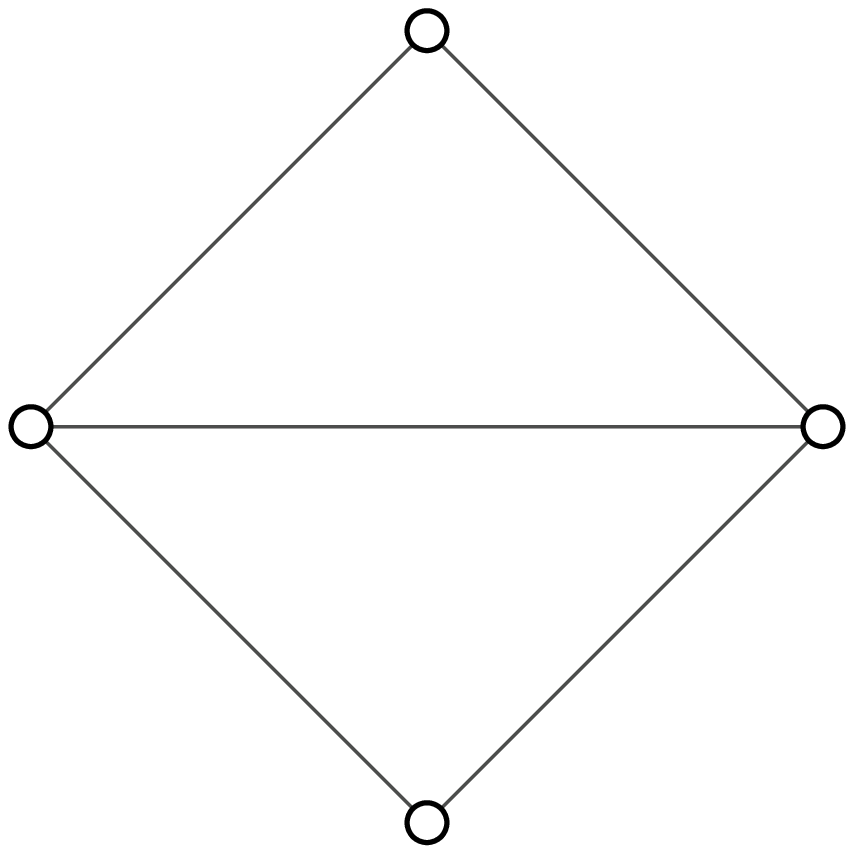} $\qquad$
\includegraphics[scale=0.3]{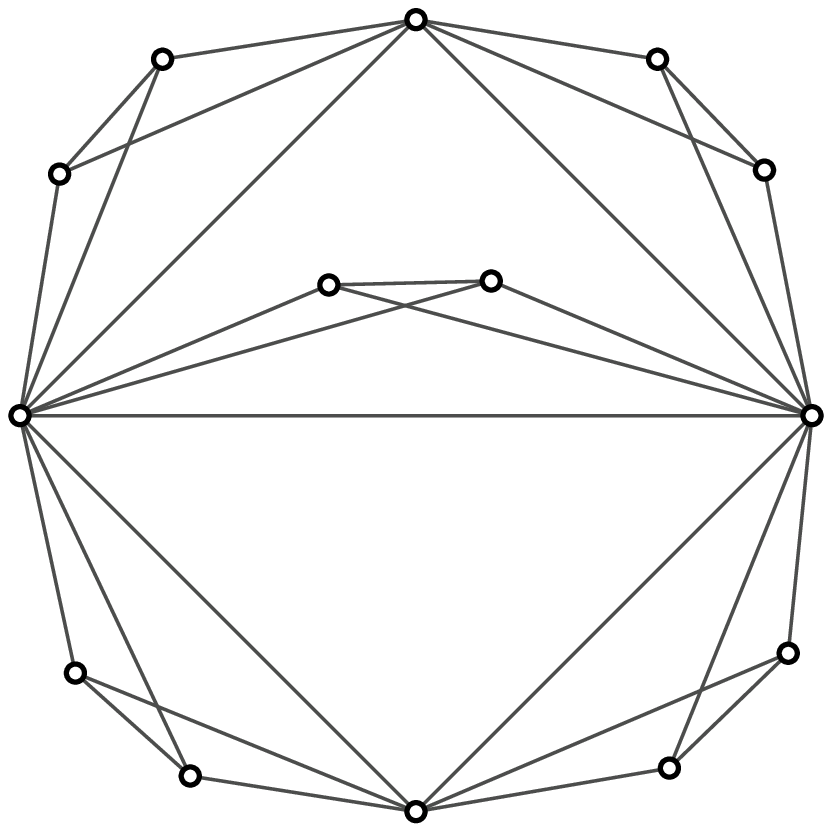}\\
\caption{The left graph is $G$ and the right graph is $G \diamond K_2$} \label{fig1}
}
\end{figure}

It follows from definition of the edge corona product that for two graphs $G$ and $H$ 
with $|V(G)|=n_1$, $|E(G)|=m_1$, $|V(H)|=n_2$ and $|E(H)|=m_2$, the graph 
$G \diamond H$ has $n_1+m_1n_2$ vertices and $m_1(1+m_2+2n_2)$ edges.
This fact allows us to recognize immediately that the associative law never holds,
 that is $G_1 \diamond (G_2 \diamond G_3)\ncong (G_1 \diamond G_2) \diamond G_3$.
  Indeed the first one has $n_1+m_1n_2+m_1m_2n_3+2m_1n_2n_3+m_1n_3$ vertices 
  while the second graph has $n_1+m_1n_2+m_1m_2n_3$ vertices and these numbers 
  are never equal. So, the associative law never holds for generalized edge corona as well.

A \emph{$k-$distance coloring} of a graph $G$ is a vertex coloring of $G$ such that no 
two vertices lying at distance less than or equal to $k$ in $G$ are assigned the same color 
\cite{Fer, Geo, Gon}. The \emph{$k-$distance chromatic number} of $G$ is the minimum
 number of colors necessary to $k$-distance color $G$, and is denoted by $\chi_{\leq k}(G)$.
  We note that proper coloring is a particular case of $k-$distance coloring, where $k = 1$ 
  and the chromatic number in this case denoted by $\chi(G)$.

A set $D$ of vertices in a graph $G$ is a \emph{dominating set} if every vertex in $V(G)- D$ is
adjacent to at least one vertex in $D$.
The \emph{domination number} $\gamma(G)$ is the number of vertices in a smallest 
dominating set for $G$. Dominating sets in graphs are natural models for facility location
 problems in operational research. These problems are concerned with the location of 
 one or more facilities in a way that optimizes a certain objective function such as 
 minimizing transportation cost, providing equitable service to customers and 
 capturing the largest market share \cite{E, Gon, Qua}.

In this paper,  the chromatic number of the generalized edge corona product of
graphs are computed. Furthermore, the upper and lower bounds for $2-$distance
 chromatic number and some results regarding the case $k=3$ have been obtained. 
  As a consequence, these quantities for the edge corona product of some graphs are
   calculated. In addition,  the dominating set, the domination number and the 
   independence number of generalized edge corona product of any connected 
   graph $G$ and arbitrary graphs $H_i$, $1 \leq i \leq |E(G)|$, are obtained.

\section{$k-$Distance Chromatic Number of $G\diamond (H_1, ..., H_m)$}
Let $G$ be a graph with vertex set $V (G)=\{v_1, v_2, ... , v_n\}$ and edge set
$E(G)=\{e_1, e_2, ... , e_m\}$. We start this section by giving the result relative 
to 1-distance chromatic number (chromatic number) of $G \diamond (H_1, ..., H_m)$.

\begin{theorem}\label{Thm:Col}
The chromatic number of $G \diamond (H_1, ..., H_m)$ is given by
$$ \chi(G \diamond (H_1, ..., H_m)) =max_{1\leq i\leq m} \{\chi(G),\  \chi(H_i)+2\}. $$
\end{theorem}
\begin{proof}
Clearly $\chi(G \diamond (H_1, ..., H_m)) \geq \chi(G)$. Also,
since every vertex $u \in V(H_i)$, $1\leq i\leq m$, is adjacent to
 both ends of $e_i=v_iv_{i+1}$, we have $\chi(G\diamond (H_1, ..., H_m))
  \geq max \{\chi(H_i)+2, 1\leq i\leq m\}$ and so, 
  $\chi(G \diamond (H_1, ..., H_m)) \geq max\{\chi(G), \chi(H_i) + 2\}$.

We now consider the reverse inequality, let $t = max_{1\leq i\leq m}
\{\chi(G), \chi(H_i) + 2\}$ and let us color the vertices $v_i$,
$1\leq i\leq n$, of $G$ with $t$ different colors $c_1, c_2, ..., c_t$. If $v_i$
 and $v_{i+1}$ have colors $c_i$ and $c_{i+1}$, respectively, then the 
 graph $H_i$ can be colored by using the set of colors $\{c_1, c_2, ..., c_{i-1}, 
 c_{i+2}, ..., c_t\}$. Hence, we arrived at $$\chi(G\diamond (H_1, ..., H_m)) 
 \leq t = max_{1\leq i\leq m} \{\chi(G),\ \chi(H_i) + 2\}$$ and the proof is completed.
 \end{proof}
 Our next step is to study the particular case $k = 2$. In this case, 
we will give the exact value of $\chi_{\leq 2}(T \diamond (H_1, ..., H_{|E(T)|}))$
 and an upper bound for $\chi_{\leq 2}(K_n \diamond (H_1, ..., H_{|E(K_n)|}))$.
  As a direct consequence,   $\chi_{\leq 2}(T \diamond H)$ and
   $\chi_{\leq 2}(K_{n_1} \diamond H)$ will be determined.
Following Kishore and  Sunitha \cite{Ki}, a tree $T$ on $n$ 
vertices can be considered as composed of $r$ branches or
$r$ star graphs connected together to form a single component 
either by edges or clinging together.

\begin{theorem} \label{T}
Let $T$ be a tree with a vertex $v$ of maximum degree 
$\triangle$, $m$ edges and $n$ vertices. If $H_{i_k}$ 
are corresponding graph to the edge $e_{i_k}=vv_i, 1\leq k\leq \triangle$,
 in $T \diamond (H_1, ..., H_m)$ and $n_{i_k}$ be the order of 
 $H_{i_k}, 1\leq k\leq \triangle$, then
$ \chi_{\leq 2}(T \diamond (H_1, ..., H_m)) =\triangle +1
+\sum_{k=1}^{\triangle} n_{i_k}. $
\end{theorem}

\begin{proof}
Consider $T$ as a subgraph of $T \diamond (H_1, ..., H_m)$. 
We can easily observe that $2-$distance chromatic number 
of $T$ is $\triangle+1$. Let there be $r$ connected star graphs for $T$.
 We start by the vertex $v$ of the maximum degree $\triangle$ in $T$.
Obviously, all the vertices adjacent to $v$ should be colored different from $v$ 
and also from each others. Hence we need $\triangle +1$ colors. Since other 
vertices have degrees less than or equal to $\triangle$, there is no more colors
 needed. We are repeating this process till all the vertices in the $r$ branches of
  $T$ are colored. Next we consider all graphs $H_{i_k},\ 1\leq k\leq \triangle$. 
  Since all of these graphs have the same neighbor $v$, they have to be colored 
  differently. Also it is clear that no pairs of vertices of each $H_{i_k}$ could be
 colored the same. Moreover, by definition of generalized edge corona and $2-$distance 
 coloring, the vertices of each $H_{i_k}$ should be colored differently from all 
 the $\triangle+1$ colors used before. Similarly, since other vertices of $T$ have 
 degrees less than or equal to $\triangle$, there is no more colors needed to color 
 all the vertices of $H_i$'s in the $r$ branches of $T$.
 So in general, we arrived at $\chi_{\leq 2}(T \diamond (H_1, ..., H_m)) 
 =\triangle+1+\sum_{k=1}^{\triangle} n_{i_k}$ and the theorem is proved.
\end{proof}

\begin{cor}
\label{Tr}
Let $T$ be a tree with maximum degree $\triangle$ and $H$ be a graph of order $n$. 
Then $$ \chi_{\leq 2}(T \diamond H) =(n+1)\triangle+1. $$
\end{cor}

A set of edges in a graph $G$ is called independent or a matching if
no two edges have a vertex in common.  The size of  any largest matching
in $G$  is called the matching number of $G$ and is denoted by $\nu(G)$. 

Before proving Theorem \ref{kn}, we record the following simple lemma:

\begin{lemma}
\label{In}
The number of independent edges in $K_n$ is
$$\nu (K_n)=\left\{{{n\over 2} \hskip 1.5cm 2\mid n}\atop {{n-1 \over 2} \hskip 1cm 2\nmid n }\right.$$
\end{lemma}

\begin{theorem}
\label{kn}
Let $K_{n}$ be a complete graph of order $n$ and $H_i$ be graphs 
with $n_i$ vertices, $1\leq i \leq |E(K_n)|$. Then,
$$ \chi_{\leq 2}(K_{n} \diamond (H_1, ..., H_{|E(K_n)|}) )
\leq \left \{{n(t+1)\atop n(t+1)-t}  \hskip 1cm {2\nmid n\atop { 2\mid n }}\right.  $$
where $t=max_{1\leq i \leq |E(K_n)|} \{n_i\}$.
\end{theorem}

\begin{proof}
It is clear that we need $n$ colors for coloring the vertices of $K_{n}$.
 Since all pairs $u$, $v$ of vertices of $H_i, \ 1\leq i \leq |E(K_n)|,$ can
  be connected with at least a path of length two, using a vertex of $K_{n}$, 
  every pair of distinct vertices of $H_i$ should be colored differently. 
  Let $t=max_{1\leq i \leq |E(K_n)|}\{n_i\}$. By Lemma  \ref{In},
   if $n$ is even, then the set of nonempty graphs $H_i$ can be divided into
    $n-1$ classes of size $n/2$ which their vertices could be colored using the
     same set of colors, and  if $n$ is odd, then we have $n$ classes of size
      ${n-1\over 2}$ of $H_i$ which could be assigned the same colors. 
      So, in general, if $n$ is even we need at most $t(n-1)$ new colors
       and if $n$ is odd, we need at most $tn$ new colors and the result follows.
\end{proof}

 Applying Theorem \ref{kn}, we obtain the following corollary.

\begin{cor}
\label{k_n}
Let $K_{n_1}$ be a complete graph of order $n_1$ and $H$ be
 a graph with $n_2$ vertices. Then we have
$$ \chi_{\leq 2}(K_{n_1} \diamond H) =\left \{{n_1(n_2+1)\qquad 
\atop n_1(n_2+1)-n_2}  \hskip 1cm {2\nmid n_1\atop { 2\mid n_1 }}\right.  $$
\end{cor}

\begin{theorem}
\label{uplow}
Let $G$ be a graph with a vertex $v$ of maximum degree $\triangle$ 
and $H_{i_k}$ be corresponding graph to the edge $e_{i_k}=vv_i, 1
\leq k\leq \triangle$, in $G \diamond (H_1, ..., H_m)$ and $n_{i_k}$
 be order of $H_{i_k}, 1\leq k\leq \triangle$. Then,
$$\triangle+1+\sum_{k=1}^{\triangle} n_{i_k} \leq \chi_{\leq 2}
(G \diamond  (H_1, ..., H_{|E(G)|}) \leq n(t+1), $$
where $|V(G)|=n$ and $t=max_{1\leq i \leq |E(G)|} \{|V(H_i)|\}$.
\end{theorem}

\begin{proof}
By taking into account that the maximum degree of $G\diamond  
(H_1, ..., H_{|E(G)|})$ is $\triangle+\sum_{k=1}^{\triangle} n_{i_k}$
 and the fact that
$\triangle(G \diamond  (H_1, ..., H_{|E(G)|}) )+1\leq \chi_{\leq 2}(G 
\diamond  (H_1, ..., H_{|E(G)|}) )$, we obtain the lower bound.

On the other hand, suppose $H_i$ be a graph corresponding to 
the edge $e_i$ of $G$. Clearly, every pair of vertices of $H_i$ 
should be colored differently. In the worst case, if $G$ is a 
complete graph, then there is a path of length two between 
every vertex of $H_i$ and every vertex of $G$. Therefore, 
the upper bound is obtained by applying the Theorem \ref{kn}.
\end{proof}

By considering Theorems \ref{T} and \ref{kn}, 
we can conclude that the above inequalities, are sharp.

\begin{cor}
Let $G$ be a graph with maximum degree $\triangle$. Then,
$$(n_2+1)\triangle+1 \leq \chi_{\leq 2}(G \diamond H) \leq  n_1(n_2+1)$$
where $n_1$ and $n_2$ are the number of vertices of $G$ and $H$, respectively.
\end{cor}

By considering Corollaries \ref{Tr} and \ref{k_n}, we can conclude 
that the above inequalities, are sharp. We now present some results
 on $3-$distance chromatic number of some families of graphs.

\begin{lemma}
\label{D}
$D(G \diamond  (H_1, ..., H_{|E(G)|}))\leq D(G)+2.$
\end{lemma}

\begin{theorem}
\label{Col2}
The $3-$distance chromatic number of $K_n\diamond (H_1, ..., H_{|E(K_n)|})$ is
 $$ \chi_{\leq 3}(K_n \diamond (H_1, ..., H_{|E({K_n})|})) =
 n+\sum_{i=1}^{{n(n-1)} \over 2} n_i.  $$
\end{theorem}

\begin{proof}
By Lemma \ref{D}, the diameter of $K_n \diamond (H_1, ..., H_{|E({K_n})|})$
 is at most three. So all vertices of $K_n \diamond (H_1, ..., H_{|E({K_n})|})$ 
 have to be assigned different colors.
\end{proof}

The next corollary now follows directly from Theorem \ref{Col2}.

\begin{cor}
The $3-$distance chromatic number of $K_n\diamond H$ is
 $$ \chi_{\leq 3}(K_n \diamond H) =n+{{n(n-1)} \over 2} |V(H)|.  $$
\end{cor}

\begin{theorem}
\label{Thm:Col2}
$ \chi_{\leq 3}(K_{m,n} \diamond  (H_1, ..., H_{mn})) =m+n+ \sum_{i=1}^{mn} n_i,$
where $n_i$ is the number of vertices of $H_i$. In particular,  
$ \chi_{\leq 3}(K_{m,n} \diamond H) =m+n+ mn n_2,  $
where $n_2$ is the number of vertices of $H$.
\end{theorem}

\begin{proof}
The result follows from the fact that $D(K_{m,n} \diamond  (H_1, ..., H_{mn}))\leqslant 3$.
\end{proof}

\section{Domination and Independence Numbers of $G\diamond (H_1, ..., H_m)$}
Let $G$ be a graph with $|E(G)|=m$. The aim of this section is to compute the
 domination and independence numbers of $G\diamond (H_1, ..., H_m)$.  
 To present this result, we need some definitions. Recall that a \emph{vertex cover}
  of a graph $G$ is a set of vertices such that each edge of $G$ is incident to 
  at least one vertex of the set. The minimum cardinality of this set is the 
  \emph{vertex covering number} which is denoted by $\beta (G)$. 
  For every $e_i=u_iv_i \in E(G)$, $1 \leq i \leq m$, by $e_i+H_i$ 
  we denote the subgraph of $G \diamond (H_1, ..., H_m)$ which obtained
   by joining two ends of the edge $e_i$ with all vertices of $H_i$, 
   where $H_i$ is a graph corresponding to the edge $e_i$ of $G$.

\begin{theorem}
\label{Thm:a}
Suppose  $G$ is a  connected graph with $m$ edges and  
$H_i, 1\leq i\leq m$ are graphs. Then $D\subseteq V(G \diamond (H_1, ..., H_m))$
 is a dominating set of $G \diamond (H_1, ..., H_m)$ if and only if $V(e_i+H_i)\cap D$
  is a dominating set of $e_i+H_i$ for every $e_i\in E(G)$.
\end{theorem}

\begin{proof}
Let $D$ be a dominating set in $G \diamond (H_1, ..., H_m)$ and
 $e_i=u_iv_i\in E(G)$. If $u_i$ (or $v_i$) belongs to $D$, then 
 $\{u_i\}$ (or $\{v_i\}$) is a dominating set of $e_i+H_i$,
  i.e. $V(e_i+H_i)\cap D$ is a dominating set of $e_i+H_i$. 
  Suppose that $D$ consists of no vertices of $u_i$ and $v_i$.
   Let $x\in V(e_i+H_i)\setminus \{D\cup \{u_i, v_i\}\}$.
Since $D$ is a dominating set of $G \diamond (H_1, ..., H_m)$,
 there exists $y\in D$ such that $xy \in E(G \diamond (H_1, ..., H_m))$. 
 Clearly, $y\in V(H_i+e_i)\cap D$ and $xy\in E(H_i+e_i)$.
  It follows that $V(H_i+e_i)\cap D$ is a dominating set of $e_i+H_i$.

Conversely, suppose that $V(H_i+e_i)\cap D$ is a dominating set of 
$e_i+H_i$ for every $e_i\in E(G), 1\leq i\leq m$. Since $G$ is connected,
 it immediately concludes that $D$ is a dominating set of $G \diamond (H_1, ..., H_m)$.
\end{proof}

\begin{theorem}
\label{Thm:b}
Let $G$ be a connected graph and $H_i, 1 \leq i\leq |E(G)|$ be arbitrary graphs. Then
$\gamma(G \diamond (H_1, ..., H_m))= \beta(G).$
\end{theorem}

\begin{proof}
Let $C$ be a vertex cover of $G$. By definition, 
we can observe that the intersection of
two of the sets $V(H_i+e_i)$ and $C$ is $\{u_i\}$ or $\{v_i\}$ or
 both of them. All the previous cases are dominating set of $H_i+e_i$.
  By Theorem \ref{Thm:a}, $C$ is a dominating set of $G \diamond H$. 
  So $\gamma(G \diamond (H_1, ..., H_m))\leq |C|=\beta(G)$.

To complete the proof, let $D^*$ be a minimum dominating set of $G \diamond H$. 
By Theorem \ref{Thm:a}, $V(H_i+e_i)\cap D^*$ is a dominating set of $H_i+e_i$
 for every $e_i\in E(G)$. Therefore, $D^*$ contains either at least one of the vertices
  $u_i$ and $v_i$ or exactly one vertex $x$ of $H_i$. If $x\in D^*$,
   then without loss of generality, we replace $x$ by $u_i$ in $D^*$. 
   By continuing this process, we reach a new minimum dominating set
    of $D^*_{new}$, whose all members belong to $V(G)$. It is clear 
    that $D^*_{new}$ is also a vertex cover of $G$. Thus, 
    $\gamma(G \diamond (H_1, ..., H_m))=|D^*_{new}|\geq \beta(G)$. 
    Therefore, $\gamma(G \diamond (H_1, ..., H_m))= \beta(G)$.
\end{proof}

\begin{cor}
Let $G$ be a connected graph and $H$ be any graph. Then  the 
domination number of $G\diamond H$ is $\beta(G)$.
\end{cor}

A set $S$ of vertices in a graph $G$ is an \emph{independent set} if and only
 if there is no edge in $E(G)$ between any two vertices in $S$. A maximum
  independent set is an independent set of largest possible size for  $G$. 
  This size is called the \emph{independence number} of $G$, and denoted $\alpha(G)$.

\begin{theorem}
For a connected graph $G$ and arbitrary graphs $H_i, 1\leq i\leq m$, 
$\alpha (G\diamond (H_1, ..., H_m))= \sum_{i=1}^m \alpha (H_i).$
\end{theorem}

\begin{proof}
Let $H_i$ be a graph corresponding to the edge $e_i=uv$. Clearly all 
vertices of $H_i$ are independent from all vertices of $H_j$, $i\neq j$. 
Let $S_{H_i}$ be a maximum independent set of $H_i$.  
Since the generalized edge corona operation on $G$ and all $H_i$ 
does not make any edge between two vertices of $S_{H_i}$, $S=\cup_{i=1}^m  S_{H_i}$
 is an independent set of $G\diamond (H_1, ..., H_m)$. 
 So $\alpha (G\diamond (H_1, ..., H_m))\geq  \sum_{i=1}^m\alpha (H_i)$.

On the other hand, let $S$ be a maximum independent set of
 $G\diamond (H_1, ..., H_m)$. If $S$ contains no vertices of $G$,
  then the proof is complete. Otherwise, consider $e_i=uv \in E(G)$
   such that $v\in S$. Clearly, in this case $\alpha (H_i)=1$ 
   and so we can remove $v$ from $S$ and add an arbitrary 
   vertex of $H_i$ to $S$. This procedure is
continued until all elements of $S$ belong to $V(H_i)$. 
So, $\alpha (G\diamond (H_1, ..., H_m))\leq   \sum_{i=1}^m \alpha (H_i)$, 
which will complete the proof.
\end{proof}

\begin{cor}
For a connected graph $G$ and an arbitrary graph $H$, 
the independence number of $G\diamond H$ is
$\alpha (G\diamond H)= m \alpha (H),$
where $m$ is the number of edges of $G$.
\end{cor}

\section{Conclusion}
 We determined some invariants of the generalized edge corona product of some graphs such as:
 \\
  the chromatic number as a particular case of $k$-distance
   chromatic number, dominating set, domination number
    and the independence number of generalized 
  edge corona product of graphs. 
  
  Furthermore, the bounds for $2-$distance chromatic 
  number and some results regarding the case $k=3$ were obtained.  
  As a consequence, these invariants for the edge corona product of two graphs
   are calculated. These experiences can help
us to reduce the problem of computing properties of big graphs 
 to the problem of computing some parameters of the factor graphs.

\smallskip

%{\bf Acknowledgment.} The authors are indebted to the referees for some helpful
% remarks leaded us to correct and improve this paper. 

%\smallskip

\end{document}